\documentclass[a4paper,12pt]{amsart}
\usepackage{standalone}
\usepackage{amsfonts}
\usepackage{amsthm}
\usepackage{amssymb}
\usepackage{amsmath}
\usepackage{theoremref}
\usepackage{enumerate}
\usepackage{bbm}
\usepackage{bm}
\usepackage{pgfplots}
\usepgfplotslibrary{fillbetween}
\usetikzlibrary{intersections}
\usepackage{mathtools}

\newcommand{\C}{\mathcal{C}}

\newcommand{\T}{\mathbb{T}}

\newcommand{\conj}[1]{\overline{#1}}
\newcommand{\D}{\mathbb{D}}

\newcommand{\Po}{\mathcal{P}}
\newcommand{\cD}{\conj{\mathbb{D}}}

\newcommand{\m}{\textit{m}}

\renewcommand\Re{\operatorname{Re}}

\newtheorem{mainthm}{Theorem}

\newtheorem{thm}{Theorem}[section]
\newtheorem{lem}[thm]{Lemma}

\newtheorem{cor}[thm]{Corollary}
\newtheorem{prop}[thm]{Proposition}
\theoremstyle{definition}

\theoremstyle{definition}

\newtheorem{defn}{Definition}[section]

\begin{document}
\title[Mean-square approximation by polynomials]{Revisiting mean-square approximation by polynomials in the unit disk}
\author{Bartosz Malman}
\address{Division of Mathematics and Physics, 
        Mälardalen University,
		Västerås, Sweden}
\email{bartosz.malman@mdu.se} 

   
\begin{abstract}
For a positive finite Borel measure $\mu$ compactly supported in the complex plane, the space $\Po^2(\mu)$ is the closure of the analytic polynomials in the Lebesgue space $L^2(\mu)$. According to Thomson's famous result, any space $\Po^2(\mu)$ decomposes as an orthogonal sum of pieces which are essentially analytic, and a residual $L^2$-space.
We study the structure of this decomposition for a class of Borel measures $\mu$ supported on the closed unit disk $\cD$ for which the part $\mu_\D$, living in the open disk $\D$, is radial and decreases at least exponentially fast near the boundary of the disk. For the considered class of measures, we give a precise form of the Thomson decompsition. In particular, we confirm a conjecture of Kriete and MacCluer from 1990, which gives an analog to Szegö's classical theorem.
\end{abstract}

\maketitle

\section{Introduction}

The classical and highly influential theorem of Szegö states that the divergence of a \textit{logarithmic integral}, namely \begin{equation} \label{logwDivergentIntegral}
\int_\T \log w \, d\m = -\infty,
\end{equation} is necessary and sufficient for the set $\mathcal{P}$ of analytic polynomials to be dense in the Lebesgue space $L^2(\mu)$, where $d\mu = w \, d\m$. Here $w$ is a non-negative Borel measurable function (a \textit{weight}) on the unit circle $\T = \{ z \in \mathbb{C} : |z| = 1 \}$, $m$ is the Lebesgue measure on $\T$, and $L^2(\mu)$ is the space of measurable functions $f$ satisfying the integrability condition \[ \|f\|^2_{L^2(\mu)} := \int_\T |f|^2 w \, d\m < \infty.\] 
Szegö's result has important consequences and applications in operator and spectral theory, complex analysis, and theory of stochastic processes, for instance. 

This paper continues a line of research dealing with the following question: what is the closure of $\mathcal{P}$ in $L^2(\mu)$ if the measure $d\mu = d\mu_\T = w \, d\m$ appearing in \eqref{logwDivergentIntegral} is enlarged to include also a part on the open disk $\D = \{ z \in \mathbb{C} : |z| < 1 \}$? If $\mu = \mu_\D + \mu_\T$, with the two pieces living on $\D$ and $\T$ respectively, then we expect to need more than the condition \eqref{logwDivergentIntegral} if, for instance, we want to conclude that any function in $L^2(\mu_\T)$ lies in the closure of $\mathcal{P}$ in $L^2(\mu)$.

There is an extensive literature dealing with this question, and results in this direction have interesting applications in other parts of analysis. We will review some of the previous results and applications below. The main result of the present paper is \thref{maintheoremSuperExp} below. It is a structural result on the span of the analytic polynomials in $L^2(\mu)$, in the case when the part of $\mu$ living on the unit disk decreases sufficiently fast near the boundary of $\D$. It can be seen as an analog to Szegö's theorem. 

\subsection{Background and some notation}

The set of analytic polynomials $\mathcal{P}$ consists of functions of the form $p(z) = \sum_{n=0}^N p_n z^n$ where $z \in \mathbb{C}$ is the complex variable. By $\Po^t(\mu)$ we denote the norm-closure of $\mathcal{P}$ in the usual Lebesgue space $L^t(\mu)$. Here $t$ is a finite positive number, and $\mu$ is a finite, positive Borel measure which is compactly supported in the plane. We shall for the most part be dealing with the case $t=2$ (but some of our results will be applicable in the case $t\neq 2$ also) and with measures $\mu$ supported on the closed unit disk $\cD = \{ z \in \mathbb{C} : |z| \leq 1\}$ of a particular structure. Before specializing to our intended context, we give a brief review of some facets of the general theory.

Thomson in \cite{thomson1991approximation} solved a long-standing problem: $\Po^t(\mu) \neq L^t(\mu)$ if and only if there exists an open set $U$ consisting of \textit{points of bounded evaluation} for $\mu$, in the sense that \[ |p(\lambda)|^t \leq C \int |p|^t d\mu\] holds for $p \in \mathcal{P}$ and some $C > 0$ independent of $p$ and $\lambda \in U$. More precisely, there exists a decomposition of the measure $\mu = \sum_i \mu_i$, each piece living on a different subset of the support of $\mu$, such that \begin{equation}
    \label{ThomsonTheoremDecomp} \Po^t(\mu) = \Big(\oplus_{i \geq 1} \Po^t(\mu_i)\Big) \oplus L^t(\mu_0)
\end{equation} where for the pieces corresponding to $i \geq 1$ there exists an open simply connected domain $U_i$ consisting of points of bounded evaluation for $\mu_i$. These pieces are also \textit{irreducible} in the sense that they do not contain any non-trivial characteristic functions. For a Borel set $A$, the characteristic function $1_A$ is defined, as usual, by \[ 1_A(x) = \begin{cases} 1, \quad x \in A \\ 0, \quad x \not\in A \end{cases}\] and it is \textit{trivial} for $\mu_i$ if it equals identically $0$ or $1$ almost everywhere with respect to $\mu_i$. In concrete applications, the pieces $\Po^t(\mu_i)$ are often easily identified as spaces of analytic functions on the domain $U_i$.

\subsection{The considered class of measures}

Thomson's theorem is a cornerstone of the theory of $\Po^t(\mu)$ spaces. However, it is hard to apply in specific examples appearing in the applications in which the exact identification of the pieces $\mu_i$ in \eqref{ThomsonTheoremDecomp} is necessary. Such applications are found, for instance, in Fourier analysis and the theory of the Cauchy integral operator (\cite{khrushchev1978problem}, \cite{malman2022thomson}) and also in the theory of de Branges-Rovnyak spaces (\cite{ConstrFamiliesSmoothCauchyTransforms}, \cite{DBRpapperAdem}). In these applications, the related space $\Po^2(\mu)$ splits into two pieces: $\Po^2(\mu_1)$, which is a space of analytic functions on $\D$, and a full space $L^2(\mu_0)$, with $\mu_0$ supported on the circle $\T$. This will also be the case in our context.

If $F$ is a Borel set, then by $\mu_F$ we will denote the restriction of $\mu$ to $F$. The Lebesgue space $L^t(\mu_F)$ can be regarded as a closed subspace of $L^t(\mu)$ in the obvious way. The measures $\mu$ in our study will live on the closed disk $\cD = \D \cup \T$, and be composed of two pieces $\mu = \mu_\D + \mu_\T$. The part $\mu_\T$, living on the circle $\T$, will be a general absolutely continuous measure \begin{equation}
    \label{MuTform}
    d\mu_\T = w \, d\m 
\end{equation} where $w$ is a non-negative Borel measurable function, and $d\m$ is the Lebesgue measure (arclength measure) on $\T$, normalized by the condition $\m(\T) = 1$. We assume the global divergence of the logarithmic integral \eqref{logwDivergentIntegral}, since it is known from elementary Hardy space theory that convergence in \eqref{logwDivergentIntegral} means that $\Po^t(\mu)$ is essentially the image of the Hardy space $H^t$ under a multiplication operator. In our main result we shall show how instead the \textit{local} integrability of $\log w$ shapes the structure of \eqref{ThomsonTheoremDecomp}. 

The part $\mu_\D$ living on $\D$ will be of the form \begin{equation}
    \label{MuDform}
    d\mu_\D(z) = G(1-|z|)dA(z),
\end{equation} where $G$ is some increasing function defined for $x$ between $0$ and $1$ and satisfying $G(0) = 0$, and $dA =\pi^{-1} dxdy$ is the normalized area measure on $\D$. We will assume one lower bound on the rate of decrease of $G$ as $x$ tends to zero, and one upper bound. Namely, we require that $G$ has at least exponential decay, and we express this by
\begin{equation} \label{ExpDecTag}
\liminf_{x \to 0^+} \, x \log(1/G(x)) > 0. \tag{ExpDec}    
\end{equation}
The upper bound will be an integrability condition 
\begin{equation} \label{LogLogIntTag}
\int_0^1 \log \log (1/G(x)) \, dx < \infty. \tag{LogLogInt}  \end{equation}
For the condition \eqref{LogLogIntTag} to make sense, we must of course have  that $\log(1/G(x)) > 0$, or in other words $G(x) < 1$. This we can always assume by reshaping the function $G$ on $(c, 1]$ for some positive $c$. Such an operation introduces an equivalent norm on $\Po^t(\mu)$, and does not affect any of our results. 

\subsection{Main result} Let $\mu$ be given by \begin{align}  d\mu & = d\mu_\D + d\mu_\T \label{mustructre} \\ & =  G(1-|z|)dA(z) + w \, d\m. \nonumber \end{align} Throughout the paper the measure $\mu$ will have this form. The main theorem states, roughly speaking, that the two assumptions \eqref{ExpDecTag} and \eqref{LogLogIntTag} place us in the correct range of functions $G$ for which the structure of $\Po^2(\mu)$ is determined by local integrability of the logarithm of $w$ on intervals $I$ (or, in other words, arcs) of $\T$. 

\begin{defn} Let $E \subseteq \T$ be the carrier set of the weight $w$: 
\begin{equation} \label{carrierwEdef}
E := \{ x \in \T : w(x) > 0 \}    
\end{equation} and define the \textit{residual set} of $w$ to consist of those $x \in E$ which are not contained in any (say, open) interval $I$ on which $\log w$ is integrable: \begin{equation} \label{reswdef}
F := \Big\{ x \in E: \int_I \log w \, d\m = -\infty \text{ whenever $x \in I$} \Big\}.    
\end{equation}
\end{defn}
It is not hard to verify that the residual set $F$ is (up to a set of $m$-measure zero) the complement in $E$ of a countable union of intervals on which $\log w$ is integrable. 

The following is our structure theorem for the considered class of $\Po^2(\mu)$-spaces, and the main result of the paper.

\begin{mainthm} \thlabel{maintheoremSuperExp}
Let $\mu$ have the form \eqref{mustructre}, with $G$ being an increasing function which is continuously differentiable and satisfies the conditions \eqref{ExpDecTag} and \eqref{LogLogIntTag}. Let $F \subseteq \T$ be the residual set of $w$ given by \eqref{reswdef}, and decompose $\mu$ as \[\mu = (\mu_\D + \mu_{\T \setminus F}) + \mu_F,\] where $\mu_{\T \setminus F}$ and $\mu_F$ are the restrictions of $\mu_\T$ to the sets $\T \setminus F$ and $F$, respectively. Then \begin{equation} \label{MainTheoremDecompositionEquation} \Po^2(\mu) = \Po^2(\mu_\D + \mu_{\T \setminus F}) \oplus L^2(\mu_F),\end{equation} where $\Po^2(\mu_\D + \mu_{\T \setminus F})$ is irreducible. 
\end{mainthm}

The span of the functions $G$ to which \thref{maintheoremSuperExp} applies is large, and starts roughly with $G(x) = \exp(-c/x)$ for any $c > 0$ (inclusive), and ends at $G(x) = \exp \exp(-c/x)$, which falls just short of the integrability assumption \eqref{LogLogIntTag}. This range is seen to be rather sharp by consideration of previous work of Khrushchev and Volberg, which we discuss below. An absolutely sharp dichotomy is very hard to achieve because of the various needed regularity conditions on $G$ which appear (and seem to be necessary) in the different works. However, the differentiability assumption on $G$ is introduced mainly for convenience. It can be avoided at the cost of more technical proofs and estimates, and in any case, some other regularity assumption.

We remark that the part of the theorem which asserts the inclusion $L^2(\mu_F) \subset \Po^2(\mu)$ holds also for any finite $t \neq 2$ (see \thref{FwSplittingOffProp}). The proof of irreducibility of the other piece in \eqref{MainTheoremDecompositionEquation} uses an adaptation from \cite{kriete1990mean} of a Hilbert space argument.

The piece $\Po^2(\mu_\D + \mu_{\T \setminus F})$ can be identified with a space of analytic functions on $\D$ in which the analytic polynomials are dense. The deep work of Aleman, Richter and Sundberg in \cite{aleman2009nontangential} explains the boundary behaviour, zero sets, and other properties of the functions contained in this space.

\subsection{Work of Kriete and MacCluer} The present paper is inspired by results and conjectures in \cite{kriete1990mean}, where Kriete and MacCluer study the so-called \textit{splitting problem} for measures of form similar to the ones appearing here. \textit{Splitting} is said to occur for a measure $\mu$ of the form \eqref{mustructre} if the the space $\Po^t(\mu)$ decomposes into orthogonal pieces living on $\D$ and $\T$, respectively. The method of Kriete and MacCluer is based on estimations of certain composition operators on Bergman spaces, and is much different from ours. It allowed them to establish results which shed light on what the definitive structure of $\Po^t(\mu)$ might be. 

Among other interesting results, they noted that logarithmic integrability of $w$ on an interval prohibits splitting if $G$ satisfies \eqref{LogLogIntTag} (we re-use parts of their argument in the proof of \thref{maintheoremSuperExp}), and they found also that a certain change of behaviour occurs, roughly, at \eqref{ExpDecTag}. Namely, results of \cite{kriete1990mean} indicate that behaviour of $\log w$ on intervals is key if \eqref{ExpDecTag} is satisfied, while if this limit is zero, then more complicated sets play a role. Based on these observations, they pose two conjectures (see Section 9 in \cite{kriete1990mean}). One of these conjectures is that splitting occurs if $\int_I \log w \, d\m = -\infty$ for every interval $I$ and if $G$ satisfies a slightly stronger version of \eqref{ExpDecTag}. Their condition on the pervasive non-integrability of $\log w$ is nothing more than the statement that $F = E$ in \eqref{reswdef}. Thus our main result confirms Conjecture 2 from \cite[Section 9]{kriete1990mean}.

\begin{mainthm} \thlabel{KrieteMaccluerConjectureProof}
    Assume that $G$ satisfies \eqref{ExpDecTag} and that \[ \int_I \log w \, d\m = -\infty\] for every interval $I$ of $\T$. Then \[ \Po^t(\mu) = \Po^t(\mu_\D) \oplus L^t(\mu_\T).\]
\end{mainthm}

The analogy with Szegö's theorem is clear. If we add a piece $\mu_\D$ to our measure $\mu_\T = w d\m$, then we need a stronger condition than \eqref{logwDivergentIntegral} to conclude that $L^t(\mu_\T) \subset \Po^t(\mu)$.

Note that our method allows for the conclusion even in the non-Hilbertian setting $t \neq 2$ (see \thref{FwSplittingOffProp}). Kriete and MacCluer proved \thref{KrieteMaccluerConjectureProof} in the special case when $w$ is a characteristic function and $t=2$ (see \cite[Corollary 7.1]{kriete1990mean}). 

Kriete and MacCluer studied also the setting in which the limit in \eqref{ExpDecTag} is zero. The methods developed in the present paper can be applied to this situation also, and one can reach some sharp versions of results from \cite{kriete1990mean}. They propose in \cite[Section 9]{kriete1990mean} a corresponding Conjecture 1 on conditions for splitting in this setting, expecting behavior similar to the one proved here in \thref{KrieteMaccluerConjectureProof} in the case when \eqref{ExpDecTag} holds, but where intervals are replaced by more complicated sets (more precisely Beurling-Carleson sets, defined below). It is not immediately clear if the present method can give a conclusive answer to that problem. This question, and more broadly the setting in which \eqref{ExpDecTag} is violated, is planned to be studied in a future work. The resolution of the conjecture would have implications for the theory of approximations in de Branges-Rovnyak spaces (see \cite{ConstrFamiliesSmoothCauchyTransforms}, \cite{DBRpapperAdem}).

\subsection{Theorems of Khrushchev and Volberg, sharpness of the main result} Both \thref{maintheoremSuperExp} and \thref{KrieteMaccluerConjectureProof} are close to optimal, in the sense that integrability of the logarithm of $w$ on intervals does not determine the structure of $\Po^t(\mu)$ for natural choices of weights $G$ which just barely fail to satisfy \eqref{ExpDecTag} and \eqref{LogLogIntTag}.

To be more precise, just below the cutoff \eqref{ExpDecTag} one can find the functions \begin{equation} \label{weight1khrushchev}
    G(x) = \exp(-c/x^\alpha)
\end{equation} for $\alpha \in (0,1)$ and $c > 0$. For such $G$ the structure of $\Po^2(\mu)$ can not be determined from the integrability properties of $\log w$ restricted to intervals. This is a consequence of fundamental results of Khrushchev from \cite{khrushchev1978problem}. His theory can be applied to certain measures of the form \eqref{mustructre} in which the weight $w = 1_E$ is a characteristic function of some set $E$. In fact, let $E$ a closed subset of $\T$ which has positive measure and satisfies the \textit{$\alpha$-Beurling-Carleson condition} \begin{equation} \label{alphaCarlesonSet} \sum_{\ell \in L} |\ell|^\alpha < \infty, \end{equation} where $L$ is the family of open intervals $\ell$ which form the complement of $E$ in $\T$, and $|\ell|$ is the length of $\ell$. It is possible to produce such a set which additionally contains no intervals, so that trivially the integral of $\log w = \log 1_E$ diverges over any interval. But, in contrast to \thref{maintheoremSuperExp}, Khrushchev's results can be applied to conclude that for $G$ given by \eqref{weight1khrushchev} and $d\mu = G(1-|z|)dA(z) + 1_E d\m$, the space $\Po^2(\mu)$ is irreducible. His results apply to general weights $G$ for which the limit in \eqref{ExpDecTag} is zero instead of positive, but some regularity assumptions are always needed. In particular, Khrushchev's method requires the integrability of $\log 1/G$. 

The paper \cite{khrushchev1978problem} also presents implicitly a typical application in complex function theory of a structure theorem for $\Po^2(\mu)$. For instance, assume that $G(z) = (1-|z|)^\beta$ for some large $\beta > 0$, $w = 1_E$ is the characteristic function of some set $E$, and that we are in the situation that \begin{equation}
    \label{p2muNonSplitting} \Po^2(\mu) \neq \Po^2(\mu_\D) \oplus L^2(\mu_\T),
\end{equation} i.e., splitting does not occur. Then there must exist a function \[f = f_\D + f_\T \in \Po^2(\mu_\D) \oplus L^2(\mu_\T)\] which is orthogonal to the monomials $\{z^n\}_{n=0}^\infty = \{z^n_\D + z^n_\T\}_{n=0}^\infty$ in $\Po^2(\mu_\D) \oplus L^2(\mu_\T)$. This orthogonality means that the positive Fourier coefficients $\widehat{(f_\T)}_n$ of the function $f_\T$, which lives only on $E$, are equal to $-\int_\D \conj{z}^n f_\D (z)(1-|z|)^\beta dA(z)$. Therefore, by the Cauchy-Schwarz inequality, these coefficients admit a bound \begin{equation} \label{uniSpecDecayFeq}
 |\widehat{(f_\T)}_n| \leq C \sqrt{\int_\D |z|^{2n} (1-|z|)^\beta dA(z)} \simeq \frac{1}{(1+n)^{\gamma}}, \quad n \geq 0
\end{equation} where $\gamma = (1+\beta)/2$. If $E$ would happen to be nowhere dense and contain no intervals, then the fact that $f_\T$ lives only on $E$ implies that this function is very irregular, in the sense of being far from smooth. So $f_\T$ satisfies two conflicting properties: it is irregular, and yet obeys a strong spectral decay condition indicated by \eqref{uniSpecDecayFeq}. Of course, the point is that the spectral decay is only one-sided. Khrushchev showed in \cite{khrushchev1978problem} that closed sets $E$ which support such a function are precisely the \textit{Beurling-Carleson sets}, and he did it essentially by solving a splitting problem for a class of $\Po^t(\mu)$-spaces. Beurling-Carleson sets are those closed sets which satisfy a weaker version of \eqref{alphaCarlesonSet}, namely \[ \sum_{\ell \in L} |\ell| \log(1/|\ell|) < \infty. \] For more information regarding Beurling-Carleson sets, and their applications in analysis, one can consult the recent article \cite{ivrii2022beurling}. The already mentioned Conjecture 1 of Kriete and MacCluer from \cite{kriete1990mean} is related to a definitive generalization of Khrushchev's results discussed in this paragraph, to the weighted context $w \neq 1_E$.

Sharpness at the other end, at the condition \eqref{LogLogIntTag}, follows from the work of Volberg. Essentially, if the condition \eqref{LogLogIntTag} is not satisfied by $G$, then local considerations of integrability of $\log w$ play no role at all, and the global divergence or convergence of the integral of $\log w$ on the whole of $\T$ is the only interesting parameter. Namely, $\Po^2(\mu)$ of the form \eqref{mustructre} with $G$ not satisfying \eqref{LogLogIntTag}, splits if and only if \eqref{logwDivergentIntegral} is satisfied, similarly to how Szegö's classical theorem works. As usual, a regularity condition on $G$ is necessary. The claim follows from Volberg's theorem on quasianalytic functions, presented in English for instance in \cite{vol1987summability}, and also in \cite{havinbook}. Volberg's theorem asserts that a function $h \in L^2$ which satisfies a very strong unilateral spectral decay estimate \[ |\widehat{h}_n| \leq C e^{-M(n)}, \quad n = 1,2,3 \ldots\] for some positive sequence $\{M(n)\}_{n \geq 1}$ satisfying \[ \sum_{n\geq 1} \frac{M(n)}{n^2} = \infty,\] should have a summable logarithm: $\int_\T \log |h| \, d\m > -\infty$. To derive the splitting statement from Volberg's theorem, one mimics the orthogonality and spectral decay argument in the previous paragraph where the application of Khrushchev's theorem is presented. See \cite{kriete1990mean} and references therein for further details of the proof.  

\subsection{Structure of the rest of the paper} In the preliminary Section \ref{backgroundsection} we set some conventions, and review a few background results from operator theory which will be used in the proofs. Section \ref{nonanalyticsection} forms the core of the paper and is concerned with the identification of the largest $L^t$-summand appearing in $\Po^t(\mu)$. This part contains the main technical constructions of the paper, and it ends with a proof of \thref{KrieteMaccluerConjectureProof}. In Section \ref{AnalyticSummandSection} we mainly re-use, extend and specialize some ideas and techniques already appearing in the literature, mainly coming from the Kriete and MacCluer paper \cite{kriete1990mean}, in order to prove the irreducibility of the first summand in \eqref{MainTheoremDecompositionEquation}. This completes the proof of \thref{maintheoremSuperExp}.

\subsection{Acknowledgement} The author would like to sincerely thank Adem Limani for many insightful discussions and ideas concerning the content of this work.

\section{Preliminaries and conventions}

\label{backgroundsection}

For clarity of exposition, we include in this preliminary section a brief discussion of two rather peripheral issues, and set two conventions. We also recall the structure of subspaces of $L^t$-spaces on the circle which are invariant for multiplication by analytic polynomials.

\subsection{$H^\infty$ as a subset of $\Po^t(\mu)$} The space $H^\infty$ consists of functions which are analytic in the unit disk $\D$ and which are uniformly bounded there. Any function $h \in H^\infty$ is well known to admit an extension to the circle $\T$. This extension is defined, only up to a subset of $\m$-measure zero, by the radial limits \begin{equation} \label{radiallimith}
   h(z) = \lim_{r \to 1^+} h(rz), \quad z \in \T. 
\end{equation} 
Using this extension, we can consider $h = h_\D + h_\T$ as a Borel measurable function on $\cD$ and as an element of the spaces $L^t(\mu) = L^t(\mu_\D) + L^t(\mu_\T)$ whenever the measure $\mu$ has the structure \eqref{mustructre}. To be more precise, we may choose a Borel set of full measure on $\T$ on which the radial limit in \eqref{radiallimith} exists, and we can set $h_\T$ to zero elsewhere.

It is easy to see that this extension of $h$ lies in $\Po^t(\mu)$. Indeed, the dilation $h_r(z) := h(rz)$, for $z \in \cD$ and $r \in (0,1)$, is holomorphic in a neighbourhood of $\cD$, and so can be approximated uniformly by analytic polynomials. Thus $h_r \in \Po^t(\mu)$ for any finite positive $t$. By \eqref{radiallimith} and the dominated convergence theorem it is clear that, as $r$ tends to $1$, the functions $h_r$ converge in the norm of $L^t(\mu)$ to the function $h = h_\D + h_\T$, which hence lies in $\Po^t(\mu)$. By this argument, we see that the closure of $H^{\infty}$ in $L^t(\mu)$ is the same as the closure of $\mathcal{P}$ in $L^t(\mu)$, both of these closures being equal to $\Po^t(\mu)$. In the proofs below we shall be working with the more flexible class $H^\infty$ instead of $\mathcal{P}$. 

If $\Po^t(\mu)$ is not itself irreducible, then there will exist many elements of the space which all have a common restriction to $\D$. \textit{Therefore, whenever we work with a function $h \in H^\infty$ considered as an element of a space $\Po^t(\mu)$, we always mean that the part $h_\T$ of $h$ which lives on $\T$ is defined by the radial boundary values of $h$, as above.} 

\subsection{Carrier sets of measurable functions} If $f$ is a member of $L^t(\mu_\T) = L^t(w \,d\m)$ then of course this function is only well-defined up to a set of $m$-measure zero, and thus so is the set \begin{equation} \label{fcarrierSec2} \{ x \in \T : |f(x)| > 0 \}. \end{equation} A set $E$ will be a \textit{a carrier set} of $f$ if it is a Borel set on which $|f(x)| > 0$ up to a set of $m$-measure zero, and such that $f(x) = 0$ on the complement of $E$, again up to a set of measure zero. The exact choice of a representative of the carrier set will never play a role, since all our measures on $\T$ will be absolutely continuous with respect to $\m$. Note specifically that the carrier set in \eqref{fcarrierSec2} differs from the usual \textit{support} of the measure $f d\m$.

\subsection{Multiplication-invariant subspaces of Lebesgue spaces on the circle} Consider the usual Lebesgue space $L^t(\m)$, say for $t > 1$. The structure of the lattice of invariant subspaces of the \textit{shift operator}: \[f(z) \mapsto z f(z), \quad f \in L^t(\m)\] is well-known. If $\mathcal{S} \subset L^t(\T)$ is closed and invariant for the shift (or equivalently, invariant under multiplication by the analytic polynomials in $\mathcal{P}$), then $\mathcal{S}$ equals either a space of the form \[ L^t(m_F) := \{ f \in L^t(\m) : f \equiv 0 \text{ on } \T \setminus F \}\] for some measurable subset $F \subset \T$, or it equals \[ qH^t := \{ qh : h \in H^t \} \] where $q$ is a unimodular function on $\T$ and $H^t$ is the usual Hardy space of functions in $L^t(m)$ which have a non-negative Fourier spectrum. A proof can be found in \cite[Lectures II and IV]{helsonbook}. We shall have a need for a result which is an easy consequence of this.

\begin{lem} \thlabel{beurlingwienerlemma} Let $t > 1$, $w \in L^1(m)$ be a non-negative function with carrier set $E$, and $\mu = w d\m$. Let $g \in L^t(\mu)$ be a non-zero function such that \begin{equation}
    \label{gtlogint} \int_\T \log( |g|^t w) \, d\m = -\infty.
\end{equation}  Then the smallest closed subspace $\mathcal{S} \subset L^t(\mu)$ which contains $g$ and is invariant under multiplication by analytic polynomials equals \[\mathcal{S} = \{ f \in L^t(\mu) : f \equiv 0 \text{ on } E \setminus F \} = L^t(\mu_F),\] where \[ F =  \{ x \in E : g(x) \neq 0\}\] is the carrier set of $g$. 
\end{lem}

\begin{proof}
The map $U: L^t(\mu) \to L^t(\m_E)$ which maps $f \in L^t(\mu)$ to the function $Uf = fw^{1/t}$ is a surjective isometry between the spaces. The mapping sends $\mathcal{S}$ to a closed subspace $U\mathcal{S}$ of $L^t(\m_E) \subset L^t(\m)$ which is invariant under multiplication by analytic polynomials. By the structure theorem for such invariant subspaces discussed above, $U\mathcal{S}$ equals either a space of the form $L^t(m_F)$, for some measurable subset $F \subset \T$ (in fact, $F \subset E$), or it equals $qH^t$. If we are in the second case, then since $g \in \mathcal{S}$, we have that $qh = gw^{1/t}$ for some non-zero $h \in H^t$. It is well known that non-zero $h \in H^t$ implies that $|h| = |qh|$ has an integrable logarithm, so we immediately arrive at a contradiction to \eqref{gtlogint}. It follows that $U\mathcal{S} = L^t(\m_F)$ for some measurable set $F$. It is not hard to see that $F$ must be of the form presented above.
\end{proof}

We remark also a corollary to \thref{beurlingwienerlemma} which will be useful to keep in mind.

\begin{cor} \thlabel{cor22} If $\mu$ is as in \eqref{mustructre} and $\Po^t(\mu)$ contains the characteristic function of a set $F \subset \T$, then it also contains the characteristic function of any measurable subset of $F$.
\end{cor}

\begin{proof}
If $\Po^t(\mu)$ contains $1_F$, then by \thref{beurlingwienerlemma} it contains $L^t(\mu_F)$, so in particular it contains the characteristic functions of any subset of $F$. 
\end{proof}

\section{Identifying the non-analytic summand}

\label{nonanalyticsection}

The goal of this section is to show that $L^t(\mu_F) \subset \Po^t(\mu)$, where $F$ is the residual set in \eqref{reswdef}. This will already imply \thref{KrieteMaccluerConjectureProof}.

\subsection{Reduction to a problem of real variable analysis}

We start by showing that our problem can be solved if we can construct a sequence of real-valued functions with certain properties. This reduction is detailed in \thref{reductionLemma} below. The proof is somewhat lengthy but it is fairly straight-forward, and uses only basic functional analysis, measure theory and the structure of shift invariant subspaces of $L^t(\m)$ discussed earlier in Section \ref{backgroundsection}.

For a measurable and real-valued function $f$ on $\T$ the Poisson integral $P_f$, appearing in $(iv)$ below, is defined as usual by

\begin{align*} \label{PoissonIntegralFormula}
P_f(z) & := \int_\T \frac{1-|z|^2}{|x-z|^2} f(x) \, d\m (x) \\
        & = \Re \int_\T \frac{x+z}{x-z} f(x) \, d\m (x), \quad z \in \D.
\end{align*}

\begin{lem} \thlabel{reductionLemma}
Let $G$ be a function which satisfies the condition \eqref{ExpDecTag}, $w \in L^1(m)$ be a non-negative measurable function, and $F$ be a subset of $E = \{ x \in \T : w(x) > 0 \}$ which is of positive Lebesgue measure. Assume that for each sufficiently large positive number $N$ there exists a bounded real-valued function $f_N$ defined on $\T$ and a corresponding set $A_N \subset \T$ such that the following five conditions hold:
\begin{enumerate}[(i)]
        \item $\int_\T f_N \, d\m = 0,$
        \item $f_N(x) \leq -N$ on $F \cap A_N$,
        \item $|F \setminus A_N|$ tends to zero as $N$ tends to $+\infty$,
        \item the Poisson integral $P_{f_N}$ satisfies the bound $P_{f_N}(z) \leq \frac{C}{1-|z|}$ for some $C > 0$ independent of $N$,
        \item $\int_\T \exp(f_N) w \,d\m \leq C$, for some $C > 0$ independent of $N$.
        
\end{enumerate}
In that case, we have that $L^t(\mu_F) \subset \Po^t(\mu)$, with $\mu$ as in \eqref{mustructre}, for any finite $t > 0$.
\end{lem}

\begin{proof}
Let $M = M(N)$ be a positive number lesser than $N$, which tends to $+\infty$ as $N$ tends to $+\infty$, and such that the ratio $N/M$ tends to $+\infty$ as $M$ and $N$ tend to $+\infty$. By replacing $f_N$ with $M^{-1}f_N$ we do not affect the assumptions $(i)$ or $(iii)$ in the statement of the lemma. The assumption $(ii)$ is only slightly changed: we have that $M^{-1}f_N \leq -N/M$ on $F \cap A_N$, which of course tends to zero uniformly on $F \cap A_N$, as $N$ tends to $+\infty$. However, the assumption $(iv)$ is improved to 
\begin{equation}
    \label{improvedPsnGrowthEstimate}
    P_{f_N}(z) \leq \frac{C_N}{1-|z|}
\end{equation} for a constant $C_N$ $(= C/M)$ which tends to zero as $N \to +\infty$, and $(v)$ is improved to $\exp(f_N) \in L^t(\mu_\T)$ for any $t > 0$, as long as $N$ (and consequently $M$) is large enough, with the corresponding $L^t(\mu_\T)$-norm being bounded uniformly in $N$ for any fixed $t$. So a simple re-labeling of the subscripts in the family $\{f_N\}_N$ gives us a new family for which $(i)$, $(ii)$ and $(iii)$ hold, and also the mentioned improved versions of $(iv)$ and $(v)$ are satisfied.

With this modification of the functions $f_N$, we construct the outer functions \begin{equation}
    \label{OuterFunctionsgN}
        g_N(z) = \exp\Big( \int_\T \frac{x + z }{x - z} f_N(x) d\m(x)\Big), \quad z \in \D
 \end{equation} each of which is a bounded analytic functions in the unit disk. These functions are contained in $\Po^t(\mu)$, in the sense explained in Section \ref{backgroundsection}. Moreover, by the second assumption in the lemma statement and by the well-known properties of boundary behaviour of outer functions, we also have \begin{equation} \label{gNboundaryValues}
     |g_N(x)| = \exp(f_N(x)) \leq \exp(-N)
 \end{equation}  for almost every $x \in F \cap A_N$. 
 
 Recalling the assumpions on $G$ in \eqref{ExpDecTag}, let \[ \inf_{x \in (0,1] } x \log(1/G(x)) = d > 0.\] Using this, and the inequality in \eqref{improvedPsnGrowthEstimate}, we obtain for $z \in \D$ a uniform growth estimate:
 \begin{align}     
     |g_N(z)| & = \exp\Big( P_{f_N}(z)\Big) \label{gNgrowth} \\ & \leq \exp\Big(\frac{C_N}{1-|z|}\Big) \nonumber \\ 
     & \leq \exp \Big( \frac{C_N}{d} \log(1/G(1-|z|)) \Big) \nonumber \\ 
     & = \Big(G(1-|z|)\Big)^{-C_N/d}. \nonumber 
 \end{align}
Since $C_N$ is eventually small enough to ensure the inequality \[\frac{tC_N}{d} \leq 1,\] we have from \eqref{gNgrowth} the norm bound \begin{equation} \label{gnNormBound}
     \sup_{N} \int_\D |g_N(z)|^t G(1-|z|) dA(z) < \infty.
 \end{equation} 
Let us suppose for the moment that $t > 1$. The remarks made at the beginning of the proof imply that, without loss of generality, we can assume that the functions $\exp(f_N)$ are uniformly bounded in the norm of the reflexive space $L^t(\mu_\T)$. In particular, by fixing a definitive sequence of numbers $N$ tending to $+\infty$ and passing to a subsequence, we can assume that the functions $\exp(f_N)$ converge weakly in the space $L^{t}(\mu_\T)$ to some function $g \in L^{t}(\mu_\T)$. Since this works for arbitrary $t$, we can even pass to further subsequence and obtain that $g \in L^{t_*}(\mu_\T)$ for some $t_* > t$. It is not hard to see that $g$ vanishes almost everywhere on $F$. Indeed, if $s$ is the Hölder conjugate exponent to $t$ and $r$ is any function in $L^s(\mu_\T)$ living only on the set $F$, then
\begin{align} \label{gvanishesOnFEq}
    \int_\T gr \, d\mu_\T  & = \lim_{N \to +\infty}  \int_\T g_N r \, d\mu_\T \nonumber \\ & = \lim_{N \to +\infty} \int_{F \cap A_N} g_Nr \, d\mu_\T  + \int_{F \setminus A_N} g_Nr \, d\mu_\T.
\end{align} For the first term in the limit above, we can use the Cauchy-Schwarz inequality and the boundary value equality in \eqref{gNboundaryValues} to obtain the bound 
\[ \Big\vert \int_{F \cap A_N} g_Nr \, d\mu_\T \Big\vert \leq \exp(-N) \cdot \|r\|_{L^s(\mu_\T)},\] which tends to zero as $N$ grows to infinity, while for the second term similarly we obtain \[ \Big\vert \int_{F \setminus A_N} g_Nr \, d\mu_\T \Big\vert \leq \|\exp(f_N)\|_{L^t(\mu_\T)} \cdot \Big(\int_{F \setminus A_N} |r|^s \, d\mu_\T \Big)^{1/s},\] which also tends to zero, by the absolute continuity of the finite measure $|r|^s d\mu_\T$, assumption $(iii)$ in the statement of the lemma, and the improved version of $(v)$ discussed above. Thus \eqref{gvanishesOnFEq} vanishes for any choice of $r$, and consequently $g \equiv 0$ on $F$.

The outer functions $g_N$ satisfy $g_N(0) = 1$ by the assumption $(i)$, and the estimate in \eqref{gNgrowth} implies that \[\limsup_{N \to +\infty} |g_N(z)| \leq 1.\] The estimate in \eqref{gNgrowth} also implies that the family $\{g_N\}_N$ is bounded uniformly on compact subsets of $\D$. It follows now from Montel's theorem and the maximum modulus theorem for analytic functions that we have \[\lim_{N \to +\infty} g_N(z) = 1,\] uniformly on compact subsets of the unit disk $\D$. By the norm bound in \eqref{gnNormBound} and again passing to a subsequence of the $N$, we can also assume that the analytic functions $g_N$ converge weakly in the space $\Po^t(\mu_\D)$ to the constant 1.

All in all, we see that we can ensure that some sequence of the functions $g_N \in \Po^t(\mu)$ converges weakly as $N \to +\infty$ to a function $g \in \Po^t(\mu)$ which is identically equal to $1$ in $\D$ and vanishes almost everywhere on the set $F$ on $\T$. Then the function $\tilde{g} := g-1 \in \Po^t(\mu)$ vanishes on $\D$ and it is equal to the constant $1$ almost everywhere on the set $F$. 

Recall that we have ensured that $\tilde{g} \in L^{t_*}(\mu_\T)$ for some $t_* > t$. Using this, we re-write \begin{align*} 
 \log (|\tilde{g}|^t w) & = (t/t_*)\log(|\tilde{g}|^{t_*} w^{t^*/t}) \\
& = (t/t_*)\log(|\tilde{g}|^{t_*}w) + (t/t_*)\log(w^{t_*/t - 1}) \\
& = (t/t_*)\log(|\tilde{g}|^{t_*}w) + (t/t_*)(t_*/t - 1)\log(w).
\end{align*} In the last line we see a linear combination, with positive coefficients, of the term $\log w$, which satisfies \[ \int_\T \log w \, d\m = - \infty,\] and the term $\log(|\tilde{g}|^{t_*}w)$ which is the logarithm of an $L^1(m)$ function. This term might or might not be absolutely integrable, but certainly satisfies \[ \int_\T \log(|\tilde{g}|^{t_*}w) \, d\m < +\infty\]  as a consequence of the integrability of $|\tilde{g}|^{t_*}w$. It follows that 
\begin{equation}
\label{tildewNotLogIntegrable} \int_\T \log( |\tilde{g}|^t w) \, d\m = -\infty. \end{equation} 
Since $\tilde{g}$ vanishes on $\D$, the smallest subpace $\mathcal{S}$ of $\Po^t(\mu)$ which is invariant for multiplication by analytic polynomials and which contains $\tilde{g}$ is actually a subspace of $L^t(\mu_\T)$. Then \thref{beurlingwienerlemma} applies to identify $\mathcal{S}$ as $L^t(\mu_{\tilde{F}})$, where $\tilde{F}$ is a carrier for $\tilde{g}$. Since $\tilde{g} \equiv 1$ on $F$, the proof is complete in the case $t > 1$.

An elementary argument extends the conclusion from finite $t > 1$ to $t > 0$. In the case $t \in (0,1]$, let $h \in L^\infty(m)$ live only on the set $F$. Then $h \in L^2(\mu_\T)$, so by what has already been proved there exists a sequence of analytic polynomials $\{p_n\}_n$ which converges to $h$ in the norm of $\Po^2(\mu)$ and also pointwise $\mu$-almost everywhere. Letting \[A_n := \{ x \in \cD : |p_n(x) - h(x)| > 1\}\] we note that the sequence $|p_n - h|^t 1_{\cD \setminus A_n}$ is bounded pointwise by 1 and converges $\mu$-almost everwhere to zero on $\cD$, so we have that \begin{align*}
    \lim_{n \to \infty} \int_{\cD} |p_n - h|^t d\mu & = \lim_{n \to \infty} \int_{A_n} |p_n - h|^t d\mu + \int_{\cD \setminus A_n} |p_n - h|^t d\mu \\ & \leq \limsup_{n \to \infty} \int_{A_n} |p_n - h|^2 d\mu + \int_{\cD} |p_n - h|^t 1_{\cD \setminus A_n} d\mu \\ & = 0,
\end{align*} where in the last step we used the dominated convergence theorem on the second integral. This shows that $h \in \Po^t(\mu)$. The set of $h \in L^\infty(m)$ which live only on $F$ is clearly a dense subset of $L^t(\mu_F)$, and so by taking closures we obtain $L^t(\mu_F) \subset \Po^t(\mu)$ also for $t \in (0,1]$.
\end{proof}

\subsection{A Poisson integral estimate}

Our primary tool for ensuring the critical property $(iv)$ in \thref{reductionLemma} will be the following estimate for Poisson integrals. 

\begin{lem}
\thlabel{SuperExpPoissonEstimate}
Let $\{I_j\}_j$ be a finite family of disjoint intervals on $\T$, and let $f = \sum_{j} f_j$ be a real-valued function such that
\begin{enumerate}[(i)]
        \item $f_j$ is supported on the interval $I_j$,
        \item $\int_{I_j} f_j \, d\m = 0$,
        \item $\int_{I_j} |f_j| \, d\m \leq C,$ where $C$ is some positive constant which is independent of $j$.
\end{enumerate}
    Then we have the Poisson integral growth estimate \[ P_f(z) \leq \frac{4C}{1-|z|},\] where $C$ is the constant in $(iii)$ above. 
\end{lem}
For the sake of the proof of the lemma we introduce a notation. If a real-valued function $g$ and an interval $I$ are given, then we define the \textit{variation} of $g$ over $I$ as the non-negative number \[ \text{var}(g,I) := \sup_{x \in I} g(x) - \inf_{x \in I} g(x).\]

\begin{proof} By symmetry, it will be sufficient to show that \[P_f(r) = \int_{\T} P_r \, d\m \leq \frac{4C}{1-r},\] where $P_r = P_r(x) = \frac{1-r^2}{|x-r|^2}$ is the Poisson kernel. Let \[f_j^+(x) = \max[f_j(x),0]\] and \[f_j^-(x) = -\min[f_j(x), 0].\] It of course holds that \[ \int_{I_j} P_r f_j^+ \, d\m \leq \sup_{x \in I_j} P_r(x) \cdot \int_{I_j} f_j^+ \, d\m \] and \[ \inf_{x \in I_j} P_r(x) \cdot \int_{I_j} f_j^- \, d\m \leq \int_{I_j} P_r f_j^- \, d\m,\] from which we deduce, using property $(ii)$ in the statement of the lemma \begin{align*}
\int_{I_j} P_r f_j \, d\m & \leq \sup_{x \in I_j} P_r(x) \cdot \int_{I_j} f_j^+ \, d\m - \inf_{x \in I_j} P_r(x) \cdot \int_{I_j} f_j^- \, d\m  \\ &= \text{var}(P_r, I_j) \int_{I_j} f_j^+ \, d\m.
\end{align*} 
It follows now from $(i)$ and $(iii)$ that \begin{align*}
    \int_\T P_r f \,d\m & \leq \sum_{j} \text{var}(P_r, I_j) \int_{I_j} f_j^+ \\ & \leq C \sum_{j} \text{var}(P_r, I_j).
\end{align*}
But since the intervals in the family $\{I_j\}_j$ are disjoint, and the Poisson kernel is unimodal, we can in fact deduce the estimate \begin{equation} \label{MainEstimatePsnGrowthLemma}\sum_{j} \text{var}(P_r, I_j) \leq \frac{2(1+r)}{1-r} \leq \frac{4}{1-r},   
\end{equation} see Figure \ref{fig:figure} for a visual proof of this inequality. This establishes the required growth estimate.     
\end{proof}

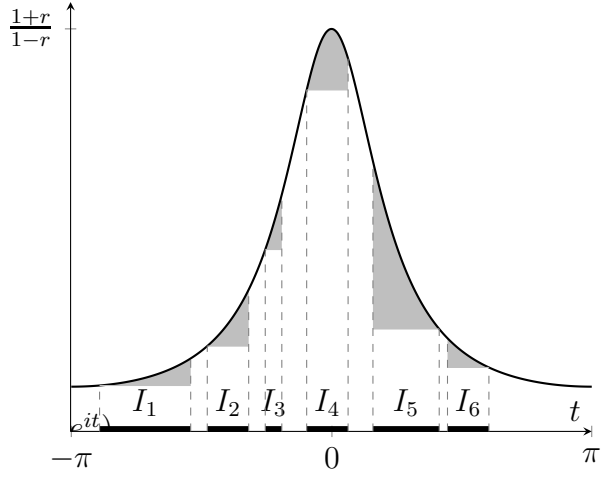
\begin{figure}
    \centering
        \begin{tikzpicture}
\begin{axis}[
        xlabel = {$t$},
        xtick={-3.1415, 0.005, 3.1415}, 
        xticklabels={$-\pi$, 0, $\pi$},
        ytick={3}, 
        yticklabels={$\frac{1+r}{1-r}$},
        xmax= 3.141,ymax=3.2, ymin=0, xmin= -3.15,        
        axis x line=middle,
        axis y line=left,
]
\addplot[name path = A,
        thick, samples= 500, domain=-pi:pi]{(1-0.5^2)/(1 - 2*0.5*cos(deg(x)) + 0.5^2)}  node at (-2,1.5) {$P_r(e^{it})$}; 
\addplot[name path = I1d, line width=4pt, domain = -0.3:0.2] {0}  node [pos=0.5, above] {$I_4$}; 
\addplot [name path = I1, color=gray!50,
    domain = -0.3:0.2] {(1-0.5^2)/(1 - 2*0.5*cos(deg(-0.3)) + 0.5^2)};
\addplot [gray!50] fill between [of = A and I1, soft clip={domain=-0.3:0.2}];
\draw [dashed, gray] (axis cs:{-0.3,0}) -- (axis cs:{-0.3,2.545});
\draw [dashed, gray] (axis cs:{0.2,0}) -- (axis cs:{0.2,2.778});
\addplot[name path = I2d, line width=4pt, domain = -1.5:-1] {0}  node [pos=0.5, above] {$I_2$}; 
\addplot [name path = I2, color=gray!50,
    domain = -1.5:-1] {(1-0.5^2)/(1 - 2*0.5*cos(deg(-1.5)) + 0.5^2)};
\addplot [gray!50] fill between [of = A and I2, soft clip={domain=-1.5:-1}];
\draw [dashed, gray] (axis cs:{-1.5,0}) -- (axis cs:{-1.5,0.636});
\draw [dashed, gray] (axis cs:{-1,0}) -- (axis cs:{-1, 1.06});
\addplot[name path = I3d, line width=4pt, domain = -0.8:-0.6] {0}  node [pos=0.5, above] {$I_3$}; 
\addplot [name path = I3, color=gray!50,
    domain = -0.8:-0.6] {(1-0.5^2)/(1 - 2*0.5*cos(deg(-0.8)) + 0.5^2)};
\addplot [gray!50] fill between [of = A and I3, soft clip={domain=-0.8:-0.6}];
\draw [dashed, gray] (axis cs:{-0.8,0}) -- (axis cs:{-0.8,1.36});
\draw [dashed, gray] (axis cs:{-0.6,0}) -- (axis cs:{-0.6, 1.77});
\addplot[name path = I4d, line width=4pt, domain = -2.8:-1.7] {0}  node [pos=0.5, above] {$I_1$}; 
\addplot [name path = I3, color=gray!50,
    domain = -2.8:-1.7] {(1-0.5^2)/(1 - 2*0.5*cos(deg(-2.8)) + 0.5^2)};
\addplot [gray!50] fill between [of = A and I3, soft clip={domain=-2.8:-1.7}];
\draw [dashed, gray] (axis cs:{-2.8,0}) -- (axis cs:{-2.8,0.342});
\draw [dashed, gray] (axis cs:{-1.7,0}) -- (axis cs:{-1.7, 0.543});
\addplot[name path = I5d, line width=4pt, domain = 0.5:1.3] {0}  node [pos=0.5, above] {$I_5$}; 
\addplot [name path = I4, color=gray!50,
    domain = 0.5:1.3] {(1-0.5^2)/(1 - 2*0.5*cos(deg(1.3)) + 0.5^2)};
\addplot [gray!50] fill between [of = A and I4, soft clip={domain=0.5:1.3}];
\draw [dashed, gray] (axis cs:{0.5,0}) -- (axis cs:{0.5,2.01});
\draw [dashed, gray] (axis cs:{1.3,0}) -- (axis cs:{1.3, 0.763});
\addplot[name path = I6d, line width=4pt, domain = 1.4:1.9] {0}  node [pos=0.5, above] {$I_6$}; 
\addplot [name path = I5, color=gray!50,
    domain = 1.4:1.9] {(1-0.5^2)/(1 - 2*0.5*cos(deg(1.9)) + 0.5^2)};
\addplot [gray!50] fill between [of = A and I5, soft clip={domain=1.4:1.9}];
\draw [dashed, gray] (axis cs:{1.4,0}) -- (axis cs:{1.4,0.694});
\draw [dashed, gray] (axis cs:{1.9,0}) -- (axis cs:{1.9, 0.477});
\end{axis}
\end{tikzpicture}
     \caption{The height of each gray area equals $\text{var}(P_r, I_j)$}
    \label{fig:figure}
\end{figure}

\subsection{Main construction and proof of the Kriete-MacCluer conjecture}

Our task on identification of $L^t$-summands has now been reduced to a construction of real-valued functions $f_N$ on $\T$ which satisfy the five assumptions appearing in \thref{reductionLemma}. 

We need one more tool before going into our main construction. We recall that a collection of open intervals $\C = \{I\}$ is a \textit{Vitali covering} of a set $F \subset \T$ if for every $x \in F$ there exists an interval $I \in \C$ which contains $x$ and is of arbitrarily short length. The following formulation of Vitali's well-known theorem will be used in our construction. A proof can be found, for instance, in \cite[pages 128-129]{stein2009real}. 

\begin{lem}\textbf{(Vitali's covering theorem)}
\thlabel{VitaliCoveringTheorem} Let $F$ be a subset of $\T$ which has positive Lebesgue measure, and let $\C = \{I\}$ be a Vitali covering of $F$. Given any $\delta > 0$, there exists a finite collection $\{I_j\}_{i=j}^n \subset \C$ consisting of pairwise disjoint intervals for which we have \[ \vert F \setminus \cup_{i=j}^n I_j \vert < \delta.\]
\end{lem}

The main result of the section is the following.

\begin{prop} \thlabel{FwSplittingOffProp}
Let $G$ be a function which satisfies the condition \eqref{ExpDecTag}, $w \in L^1(m)$ be a non-negative measurable function, and $E$ and $F$ be respectively the carrier and the residual set of $w$, as defined in \eqref{carrierwEdef} and \eqref{reswdef}. Assume that $F$ has positive $m$-measure. If $d\mu = G(1-|z|)dA(z) + w \, d\m$, then for any finite $t > 0$ we have that $L^t(\mu_F) \subset P^t(\mu)$.
\end{prop}

\begin{proof}
To start, note that the set $F$ is equal, up to a set of Lebesgue measure zero, to the countable union \[\cup_{n = 1}^\infty F_{n} = \cup_{n = 1}^\infty F \cap \{ x \in E : w(x) > 1/n\}.\] It is thus sufficient to show that for each $n$ we have the containment $L^t(\mu_{F_n}) \subset P^t(\mu)$. Consequently, we lose no generality assuming that the set $F$ is contained in \begin{equation}
     \label{FwPartAssumption}
\{ x \in E : w(x) > c\} \end{equation} for some arbitrary number $c > 0$.

Let $N > 0$ be some large number. By definition of $F$, for each point $x \in F$ there exist intervals $I$ of arbitrary small length which contain $x$ and for which we have $\int_I \log w \, d\m = -\infty.$ Moreover, it is easy to see that the family $\mathcal{C} = \mathcal{C}_N$ of all such intervals which additionally satisfy the condtition \begin{equation}
    \label{ShortIntervalVitaliCover}
    N|I| \leq 1
\end{equation} constitutes a Vitali covering of the set $F$. Extract from $\mathcal{C}$ a finite collection $\{I_j\}_j$ of pairwise disjoint intervals such that \begin{equation} \label{FminusUnionIjmeasure}
    |F \setminus \cup_j I_j| < 1/N.
\end{equation}
We proceed to show how to define \begin{equation}
    \label{fNformula} f_N = \sum_j f_{N,j}
\end{equation} where $f_{N,j}$ will live only on the interval $I_j$, so that all the five assumptions of \thref{reductionLemma} will be satisfied. Fix one of the sets $ I_j$, and recall the definition of $c$ in \eqref{FwPartAssumption}. We decompose the logarithmic integral of $w$ over $I_j$ into three pieces: 
\begin{align*}
+\infty & = \int_{I_j} \log (1/w) \, d\m \\
        & = \int_{I_J \setminus E} \log(1/w) \, d\m + \int_{ \{ w > c\} \cap I_j} \log (1/w) \, d\m \\
        & + \lim_{d \to 0^+} \int_{\{ d \leq w \leq c\} \cap I_j} \log (1/w) \, d\m.
\end{align*} The second piece satisfies \[ -\infty < \int_{ \{ w > c\} \cap I_j} \log (1/w) \, d\m < +\infty.\] Indeed, the upper bound follows from the inequality $\log (1/w) < \log (1/c)$ which holds on the set $\{ w > c\}$, and the lower bound is an easy consequence of the integrability of $w$. We see from this that at least one of the two remaining pieces must be equal to $+\infty$. Consequently, we deduce that at least one of the following situations must occur: either we have that 
\begin{equation}
    I_j \setminus E \text{ has positive Lebesgue measure} \label{case1} \tag{C1}
\end{equation}
or we have that the family of sets \[E_d := \{ x \in E : d \leq w(x) \leq c \}\] satisfies 
\begin{equation} \label{case2} \tag{C2}
  \lim_{d \to 0^+} \int_{I_j \cap E_d} \log (1/w) \, d\m = \infty.
\end{equation} Of course, it might be so that both situations occur.

The first alternative \eqref{case1} is simpler. In this case, we may set \begin{equation}
    \label{fjdef0}
    f_{N,j} = N\frac{|I_j \cap F|}{|1_{I_j \setminus E}|}1_{I_j \setminus E} - N 1_{I_j \cap F}.
\end{equation} If we are in the situation \eqref{case2}, then there exists a positive number $d$ and some measurable subset $I_j^+$ of $I_j \cap E_d$ for which we have \begin{equation} \label{Iplusestimate}
    N|I_j| \leq \int_{I_j^+} \log (1/w) \, d\m \leq 2 N|I_j| \leq 2.
\end{equation} Indeed, by \eqref{case2} there exists some $d$ small enough so that $|N|I_j \leq \int_{I_j \cap E_d} \log (1/w) \, d\m$, and then by absolute continuity of the finite measure $\log (1/w) 1_{I_j \cap E_d} d\m$ we might pick a subset $I_j^+$ so that \eqref{Iplusestimate} holds. Let \begin{equation}
    \label{alphajdef}\alpha_j := \int_{I_j^+} \log (1/w) \, d\m \simeq N|I_j|
\end{equation} and \[I_j^- := I_j \setminus I_j^+\] If \eqref{case2} holds, then we may set \begin{equation} \label{fjdef1}
    f_{N,j} = \log(1/w) 1_{I^+_j} - \frac{\alpha_j}{|I_j^-|}1_{I_j^-}.
\end{equation} If both \eqref{case1} and \eqref{case2} hold, then we are free to construct $f_{N,j}$ in either way. 

We now verify that this way of constructing $f_{N,j}$ implies that the functions $f_N$ in \eqref{fNformula} satisfy all five conditions of \thref{reductionLemma}. Both alternative definitions of $f_{N,j}$ ensure that \[ \int_{I_j} f_{N,j} \, d\m = 0,\] which means of course that \[\int_\T f_N \, d\m = \sum_{j} \int_{I_j} f_{N,j} \, d\m = 0.\] Hence $(i)$ of \thref{reductionLemma} holds. By setting $A_N := \cup_j I_j$, we see from \eqref{FminusUnionIjmeasure} that assumpion $(iii)$ of \thref{reductionLemma} also holds. Next, we verify $(ii)$. If $f_{N,j}$ was constructed according to \eqref{fjdef1}, and $x \in I_j \cap F$, then $w(x) > c$, and so $x \not\in I_j^+$. Thus from \eqref{alphajdef} we deduce that in the second case \[ -f_{N,j}(x) = \frac{\alpha_j}{|I_j^-|} \geq \frac{\alpha_j}{|I_j|} \geq N,\] while clearly $-f_{N,j}(x) = N$ if $x \in I_j \cap F$ and we constructed $f_{N,j}$ according to \eqref{fjdef0}. We have hence verified that $(ii)$ of \thref{reductionLemma} is satisfied. The assumption $(iv)$ is satisfied as a consequence of \thref{SuperExpPoissonEstimate} and \eqref{ShortIntervalVitaliCover}, since from $\int_{I_j} f_{N,j} \, d\m = 0$ we deduce \[ \int_{I_j} |f_{N,j}| \, d\m = 2 \int_{I_j} \max[f_{N,j}, 0] \, d\m\] which, in the first case, can be estimated by \[ 2 \int_{I_j} \max[f_{N,j}, 0] \, d\m = 2 N |I_j \cap F| \leq 2 N|I_j| \leq 2,\] and similarly in the second case by 
\[ 2 \int_{I_j} \max[f_{N,j}, 0] \leq 2 \alpha_j \leq 4 N|I_j| \leq 4.\]
Finally, the fifth assumption in \thref{reductionLemma} is satisfied since at any $x \in \T$ the function $\exp(f_N)$ will attain a value which is either less than or equal to $1$, or equal to $\exp(f_{N,j}(x))$ for some unique index $j$. If $f_{N,j}$ was constructed according to the first case, and so is given by \eqref{fjdef0}, then \[\exp(f_{N,j}(x))w(x) = 0\] at any point $x$ where $f_{N,j}(x) > 0$, and if it was constructed according to the second case in \eqref{fjdef1}, then \[\exp(f_{N,j}(x))w(x) = 1\]  whenever $f_{N,j}(x) > 0$. So in any case, \[ \int_\T \exp(f_N)w \, d\m \leq \int_\T (1 + w) \, d\m = 1 + \|w\|_{L^1(m)}.\] Since $N > 0$ is arbitrary, we have now verified the five requirements to apply \thref{reductionLemma}, and so the proof is complete. 
\end{proof}

As an immediate corollary, we see that we have proved the conjecture of Kriete and MacCluer from \cite[Section 9]{kriete1990mean}, in an even stronger form than asserted.

\begin{proof}[Proof of \thref{KrieteMaccluerConjectureProof}]
    The assumption that $\log w$ is not integrable on any interval clearly implies that the set $F$ in \thref{FwSplittingOffProp} coincides with the carrier set $E$ of the weight $w$, defined in \eqref{carrierwEdef}. Thus the result follows directly from \thref{FwSplittingOffProp}.
\end{proof}

\section{Identifying the analytic summand}

\label{AnalyticSummandSection}

We specialize now to the Hilbertian setting $t=2$ and show that if $I$ is any interval on which the weight $w$ is $\log$-integrable, then this interval has trivial intersection with any $L^2$-summand contained in $\Po^2(\mu)$ of the form \eqref{mustructre}. The result is more or less anticipated from the work of Kriete and MacCluer, \cite[Theorem D]{kriete1990mean}. However, we shall need a slightly more specialized statement than what appears there. Consequently, some of the techniques in the proofs below are only a minor adaptation of the ones used in their work.

\subsection{A criterion for local irreducibility} The next lemma is the key to showing irreducibility of a space.

\begin{lem} \thlabel{irreducibilityCriterionLemma}
    Assume that there exists a function \[f = f_\D + f_\T \in \Po^2(\mu_\D) \oplus L^2(\mu_\T)\] which is orthogonal to $\Po^2(\mu)$. Let $C$ be a carrier set of the part of $f$ living on $\T$: \[ C := \{ x \in \T : |f_\T(x)| > 0 \}.\] Then $\Po^2(\mu)$ does not contain the characteristic function of any measurable subset of $C$ which is of positive Lebesgue measure.
\end{lem}

\begin{proof}
    Assume, seeking a contradiction, that the characteristic function $1_B$ of the set $B \subset C$ is contained in $\Po^2(\mu)$, and $|B| > 0$. We will show that $f \equiv 0$ on $B$, which is a contradiction to $C$ being a carrier of $f_\T$. 
    
    We conclude from \thref{beurlingwienerlemma} (or from \thref{cor22}) that $1_B \in \Po^2(\mu)$ implies that $L^2(\mu_B)$ is contained in $\Po^2(\mu)$. Since the function $f$ is orthogonal to $\Po^2(\mu)$, we have that \[ \int_B f_\T \conj{g} w \, d\m = 0 \] for any function $g \in L^2(\mu_B)$. This means that \[ f w \equiv 0 \text{ on } B.\] Since the carrier set $C$ of $f_\T$ must be (up to a subset of Lebesgue measure zero) contained in the carrier set $E$ of $w$, this means that $f = 0$ almost everywhere on $B$. This is the desired contradiction.
\end{proof}

\subsection{Beurling-Malliavin majorants} 

We now want to construct a large family of functions $f$ to which \thref{irreducibilityCriterionLemma} applies. 
To this end we will find useful the following simple special case of the famous, and much more difficult to prove, general form of the Beurling-Malliavin multiplier theorem.

\begin{lem}
\thlabel{babyBMtheorem}
Let $P$ be a positive function defined on $\mathbb{R}$ which is even and decreasing for $x \geq 0$. If \begin{equation*}
\int_\mathbb{R} \frac{\log P(x) }{1+x^2} \, dx > -\infty,
\end{equation*}
then for any $a > 0$ there exists a $C^\infty$ function $g$ with support inside the interval $(-a,a)$ and with $g(0) \neq 0$, for which the Fourier transform \[ \hat{g}(x) = \int_{\mathbb{R}} e^{-ixt} g(t) dt\] satisfies the inequality \[ |\hat{g}(x)| \leq P(x), \quad x \in \mathbb{R}.\]
\end{lem}

A proof of \thref{babyBMtheorem} can be found in \cite[pages 276-277]{havinbook}. An exposition and proof of the general Beurling-Malliavin multiplier theorem is also contained in \cite{havinbook}.

\subsection{An estimate for the moments of $G$} For every analytic polynomial $p(z) = \sum_n p_n z^n$ we have that 

\begin{equation*}
    \int_{\D} |p|^2 d\mu_\D = \int_{\D} |p(z)|^2 G(1-|z|) \, dA(z) = \sum_{n} |p_n|^2 \alpha_n
\end{equation*}
where 
\begin{equation} \label{GmomentsEq}
    \alpha_n = 2 \int_0^1 r^{2n+1} G(1-r) dr
\end{equation} are the moments of the function $G(1-r)$. The rate of decay of $G(1-|z|)$ near the boundary of the disk and the rate of decay of the moments $\alpha_n$ is connected through the following lemma. 

\begin{lem} \thlabel{GmomentsLogSummability}
If $G$ is continuously differentible on $(0,1)$ and satisfies \eqref{LogLogIntTag}, then the function \begin{equation*}
        P(x) := \int_0^1 r^{|x|} G(1-r) dr, \quad x \in \mathbb{R}
    \end{equation*} satisfies \begin{equation} \label{MomentFuncPIntegrability}
        \int_\mathbb{R} \frac{\log P(x) }{1+x^2} \, dx > -\infty.
    \end{equation}
\end{lem}

Note that $P$ in the above statement is an even and decreasing function of $x \geq 0$. We also have \[ P(2n+1) = \alpha_n/2.\] 

The lemma follows from a combination of two arguments in \cite{beurling1972analytic} and \cite{havinbook}. 

\begin{proof}
Following Beurling in \cite{beurling1972analytic}, we set \[ m(x) := \log(1/G(x)), \quad x \in (0,1]\] and \begin{equation} \label{LegendreEnvelopeK} k(x) := \inf_{y \in (0,1]} m(y) + yx, \quad x > 0.\end{equation} Then $m$ is a decreasing, non-negative and differentiable on $(0,1)$, and $k$ is positive. Beurling proves (see \cite[Lemma 1]{beurling1972analytic}) that our assumption \eqref{LogLogIntTag} on $G$ implies that \begin{equation} \label{kintegrability} \int_0^\infty \frac{k(x)}{1+x^2} dx < \infty.\end{equation}
Now, we follow \cite[Page 229]{havinbook}. In the reference, the context is somewhat different and some additional assumptions on $m$ are present, but they are not needed for our purposes. For all sufficiently large positive $x$, we will now estimate $P(x)$ from below in terms of $k(x)$. We have \begin{align}
    P(x) & = \int_0^1 r^x G(1-r) \,dr \label{PxEstimate} \\
        & = \int_0^1 \exp\Big( x \log r - m(1-r)\Big) dr \nonumber \\
        & \geq \int_0^1 \exp\Big( 2x(r-1) - m(1-r) \Big) dr \nonumber \\
        & = \int_{0}^1 \exp\Big(-2xt - m(t)\Big) dt.  \nonumber
\end{align}
Since $\lim_{y \to 0^+} m(y) = +\infty$ the infimum on the right-hand side of \eqref{LegendreEnvelopeK} is attained at least at one point $y \in (0,1]$. By comparing the values of the expression in \eqref{LegendreEnvelopeK} for $y = 1/2$ and $y = 1$, we see that for sufficiently large $x$, the the infimum will be actually be attained inside of the interval $(0,1)$. At  any such point $y$ we must have, by basic calculus, the equality $m'(y) = -x$. But $m'(y) = -G'(y)/G(y)$, and this expression reveals that $m'(y)$ is bounded uniformly on each interval of form $[\delta,1], \delta > 0$. It follows that \[ \lim_{x \to +\infty} \sup_{m'(y) = -x} y = 0,\] so for sufficiently large $x$, there is always a point $\zeta = \zeta_x$ which satisfies $\zeta_x < 1/2$ and $k(x) = m(\zeta_x) + \zeta_x x$. Returning to \eqref{PxEstimate}, we use that $m$ is decreasing to estimate \begin{align*}
    P(x) & \geq \int_{\zeta_{2x}}^1 \exp\Big(-2xt - m(t)\Big) dt \\
    & \geq \exp\big(-m(\zeta_{2x})\big) \int_{\zeta_{2x}}^1 \exp (-2xt) dt \\
    & = \exp\Big(-m(\zeta_{2x})\Big)\frac{\exp(-2x\zeta_{2x}) - \exp(-2x)}{2x} \\
    & = \frac{\exp\big(-k(2x)\big)}{2x} \Big(1-\exp( -2x + 2x\zeta_{2x})\Big) \\
    & \geq \frac{\exp\big(-k(2x)\big)}{4x}.
\end{align*} In the last step we used that $\zeta_{2x} < 1/2$ when $x$ is sufficiently large, and consequently $1 - \exp(-2x + 2x \zeta_{2x})$ is larger than $1/2$ when $x$ is sufficiently large. Now \eqref{MomentFuncPIntegrability} follows easily from \eqref{kintegrability} and our last estimate.     
\end{proof}

\subsection{Finalizing the irreducibility proof}

Recall the definition of the residual set $F$ in \eqref{reswdef}, which is a subset of the carrier $E$ of $w$ defined in \eqref{carrierwEdef}. The set $E \setminus F$ has the property that for almost every $x$ in that set, there exists an interval $I$ containing $x$ on which $w$ is $\log$-integrable. We will now show that the space $\Po^2(\mu_\D + \mu_{\T \setminus F})$ contains no characteristic functions, i.e, that it is irreducible. Together with the result of \thref{FwSplittingOffProp}, this will essentially complete the proof of \thref{maintheoremSuperExp}. 

In the proof of the next lemma we use the same core idea as Kriete and MacCluer in their proof of \cite[Theorem D]{kriete1990mean}.

\begin{lem} \thlabel{irreducibilityLemmaLogInt} Assume that $G$ is as in \thref{GmomentsLogSummability}, and $w$ has the property that its carrier set is, up to a set of Lebesgue measure zero, a union of intervals $I$ on which we have \[ \int_I \log w \, d\m > -\infty.\] If $\mu$ is given by \eqref{mustructre}, then $\Po^2(\mu)$ is irreducible.
\end{lem}

\begin{proof}

Let $I$ be any interval on which $\log w$ is integrable, and let $x$ be an arbitrary point in the interior of $I$. We will construct a tuple $f_\D + f_\T \in \Po^2(\mu_\D) \oplus L^2(\mu_\T)$ which is orthogonal to $\Po^2(\mu)$ and such that the carrier set $C$ of $f_\T$ contains an interval $J \subset I$ centered at $x$. An application of \thref{irreducibilityCriterionLemma} then shows that $\Po^2(\mu)$ contains no characteristic functions of subsets of $J$. Since $x$ is arbitrary, we deduce from \thref{cor22} by patching together the interval $I$ with the subintervals $J$, that $\Po^2(\mu)$ contains no characteristic functions of subsets of $I$. Finally, since up tu a null set, the entire carrier of $w$ can be covered by intervals $I$ on which $\log w$ is integrable, it will follow that $\Po^2(\mu)$ is irreducible.

We proceed to construct such a tuple. Without loss of generality we can assume that $x = 1$ and that $I$ is some interval centered at $1$. Let $P$ be as in \thref{GmomentsLogSummability} and set \[ \Tilde{P}(x) = 2 \frac{P(2|x|+1)}{1 + |x|}.\] Then it is not hard to see that $\tilde{P}$ satisfies the hypothesis of \thref{babyBMtheorem}, and so there exists a smooth function $g$ supported on an interval $(-a,a)$ of $\mathbb{R}$ which satisfies $g(0) \neq 0$ and has a Fourier transform obeying the estimate \[ |\hat{g}(x)| \leq \Tilde{P}(x).\] If $h$ is defined on the circle by the equation \[h(e^{it}) = g(t), \quad t \in (-\pi, \pi),\] then $h$ does not vanish at $1 \in \T$, is supported in an interval $J$ around $1$ which we might suppose is contained in $I$ (since support of $g$ can be chosen arbitrarily short), and it has a Fourier series $h = \sum_{n} h_ne^{itn}$ which satisfies \begin{equation}
\label{hfourierseriesestimate}
 |h_n| = |\hat{g}(n)| \leq \tilde{P}(n) = \frac{\alpha_n}{1+n},\end{equation} where $\alpha_n$ is the $n$:th moment of $G$ defined in \eqref{GmomentsEq}. Let $F(z) = \sum_{n=0}^\infty F_n z^n$, where \[F_n = -\frac{h_n}{\alpha_n}.\] Then the coefficients $F_n$ satisfy $|F_n| \leq \frac{1}{1+n}$, so that $F$ is an analytic function in $\D$. The computation \begin{align*}
     \int_\D |F(z)|^2 G(1-|z|) dA(z) & = \sum_{n \geq 0} |F_n|^2 \alpha_n \\ & \leq \sum_{n \geq 0} \frac{|\alpha_n|}{(1+n)^2} < \infty
 \end{align*} shows that $f$ is a member of $\Po^2(\mu_\D)$, since the moments $\alpha_n$ are certainly bounded. The Taylor coefficients of $F$ have been chosen in such a way that \[ \int_\D F(z) \conj{z}^n G(1-|z|) dA(z) = F_n \alpha_n = -h_n =  -\int_\T h \conj{z}^n d\m\] for non-negative integers $n$. This means precisely that the monomials $\{z^n\}_{n \geq 0} = \{z^n_\D + z^n_\T\}_{n \geq 0}$ are orthogonal to $F + h$, considered as an element of $\Po^2(\mu_\D) \oplus L^2(\m)$. Consequently the closed linear span of these monomials, which of course is $\Po^2(\mu_\D + \m)$, is orthogonal to $F+h$. The assumption that $\int_I \log w \, d\m > -\infty$ implies that we can construct a function $u \in H^\infty \subset \Po^2(\mu_\D + \m)$ which satisfies $|u| = \min(w,1)$ on the interval $I$. Then $z^n u$ is orthogonal to $F + h$ in $\Po^2(\mu_\D) \oplus L^2(\m)$. By boundedness of $u$ we have that $F\conj{u}$ is a member of $L^2(\mu_\D)$, and we can orthogonally project it onto $f_\D \in \Po^2(\mu_\D)$. We then compute 
 \begin{align*}
 0 & = \int_\D F\conj{u}\conj{z}^n G(1-|z|) dA + \int_\T h\conj{u}\conj{z}^n d\m \\ & = \int_\D f_\D \conj{z}^n G(1-|z|) dA + \int_I h\frac{\conj{u}}{w} \conj{z}^n w d\m.
 \end{align*} This means that if we set $f_\T := h\frac{\conj{u}}{w}1_I$, then $f_\T$ is a bounded measurable function, and the tuple $f_\D + f_\T  \in \Po^2(\mu_\D) \oplus L^2(\mu_\T)$ is orthogonal to $\Po^2(\mu)$. Moreover, $f_\T$ is non-zero almost everywhere on $J$. The proof is complete, by the remarks made in the first paragraph of the proof.

\end{proof}

\begin{proof}[Proof of \thref{maintheoremSuperExp}]

By \thref{FwSplittingOffProp}, the space $L^2(\mu_F)$ is contained in $\Po^2(\mu)$. The orthogonal complement of this space in $\Po^2(\mu)$ is clearly $\Po^2(\mu_\D + \mu_{\T \setminus F})$. By \thref{irreducibilityLemmaLogInt}, the space $\Po^2(\mu_\D + \mu_{\T \setminus F})$ is irreducible.
\end{proof}

\bibliographystyle{plain}
\bibliography{mybib}

\end{document}